\documentclass[11pt]{amsart}

\usepackage{amsmath,amssymb,amsthm}
\usepackage[margin=1.25in]{geometry}

\usepackage{array}
\usepackage{booktabs}
\usepackage{makecell}
\usepackage[table]{xcolor}
\usepackage{url}

\newtheorem{theorem}{Theorem}
\newtheorem{proposition}[theorem]{Proposition}
\newtheorem{lemma}[theorem]{Lemma}
\newtheorem{example}[theorem]{Example}

\theoremstyle{remark}
\newtheorem*{remark*}{Remark}

\renewcommand{\arraystretch}{1.08}

\newcommand{\Q}{\mathbb{Q}}
\newcommand{\R}{\mathbb{R}}
\newcommand{\Z}{\mathbb{Z}}
\newcommand{\sfpart}{\operatorname{sf}}

\title{Squared edge lengths of regular simplices with rational vertices}

\author[Scott Duke Kominers]{Scott Duke Kominers}
\address{Harvard Business School; Department of Economics and Center of
Mathematical Sciences and Applications, Harvard University; and a16z crypto}
\email{kominers@fas.harvard.edu}
\thanks{I used LLMs to assist with the invariant computations in the
preparation of this article, particularly GPT-5.4 Pro and Claude 4.6 Opus
(accessed in part via Poe with the support of Quora, where I am an advisor).
I particularly appreciate a thorough review from Refine.ink. The problem,
methods, and eventual written form are my own; and of course any errors remain
my responsibility. This work was conducted while I was visiting the
Technological Innovation, Entrepreneurship, and Strategic Management (TIES)
Group at the MIT Sloan School of Management; I greatly appreciate their
hospitality.}

\begin{document}

\begin{abstract}
We determine exactly which positive rational numbers occur as squared
edge lengths of regular \(d\)-simplices with vertices in \(\Q^n\).
The answer exhibits a sharp stabilization phenomenon: once \(n-d\ge 3\),
every positive rational number occurs, while codimensions \(0\), \(1\),
and \(2\) are governed by explicit square-class, norm-group, and
Hilbert-symbol conditions. The proof reduces simplex realizability to the 
Hasse--Minkowski classification of rational quadratic forms.
\end{abstract}

\maketitle
\markboth{REGULAR SIMPLICES WITH RATIONAL VERTICES}{REGULAR SIMPLICES WITH RATIONAL VERTICES}

\section{Introduction}

We determine exactly which positive rational numbers occur as squared edge
lengths of regular simplices with rational vertices.

For integers \(1\le d\le n\), define
\[
S_{\Q}(d,n)
:=
\left\{
m\in\Q_{>0} :
\exists \text{ a regular \(d\)-simplex in \(\Q^n\) of squared edge length } m
\right\}.
\]
Our main theorem gives a complete description of \(S_{\Q}(d,n)\) for every
pair \((d,n)\).

The answer exhibits a sharp stabilization phenomenon. Writing \(c=n-d\) for
the codimension, we prove that
\[
c\ge 3
\quad\Longrightarrow\quad
S_{\Q}(d,n)=\Q_{>0}.
\]
Thus in codimension at least \(3\), there is no arithmetic obstruction: every
positive rational number occurs. All exceptional behavior is confined to
codimensions \(0\), \(1\), and \(2\), where the answer is governed by explicit
square-class, norm-group, and Hilbert-symbol conditions depending only on
\(d\bmod 4\) and on the squarefree part of \(d+1\).

The key observation is that the simplex-realizability problem is equivalent
to an embedding problem for quadratic forms: Let \(A_d\) denote the Gram
matrix of a regular \(d\)-simplex of squared edge length \(2\). Writing
\(m=2a\), the existence of a regular simplex of squared edge length \(m\)
in \(\Q^n\) is equivalent to the existence of a positive definite rational
quadratic form \(r\) of rank \(n-d\) such that
\[
aA_d \perp r \simeq n\cdot\langle 1\rangle.
\]
Using the identity
\[
A_d \perp \langle d+1\rangle \simeq (d+1)\cdot\langle 1\rangle,
\]
the problem reduces to determining whether there exists a positive definite
form of prescribed rank, determinant class, and local Hasse invariants.

The quadratic-form perspective explains the codimension threshold. In rank
at least \(3\), the complement has enough local flexibility to realize the
required invariants. In rank \(2\), the problem becomes a prescribed
Hilbert-symbol problem; in rank \(1\), it becomes a norm-group condition; and
in rank \(0\), one must have \(aA_d\simeq d\cdot\langle 1\rangle\) itself.
The four possibilities just described correspond exactly to the four rows of
Theorem~\ref{thm:main}.

The problem is closely related to earlier work on rational and integral
simplices. Schoenberg~\cite{Schoenberg1937} classified the dimensions in which
a regular \(d\)-simplex can be realized in the standard lattice \(\Z^d\), and
Pelling~\cite{Pelling1977} gave the corresponding rational-coordinate
formulation. Equivalently, these works determine when the codimension-\(0\) row
in our classification is nonempty; Theorem~\ref{thm:main} refines this by
identifying the precise square classes of the squared edge lengths.
Beeson~\cite{Beeson1992} characterized all triangles embeddable in \(\Z^N\);
specialized to equilateral triangles and viewed up to square class, this
recovers the \(d=2\) cases of our classification. For \((d,n)=(3,3)\), our
classification recovers the square class \(2(\Q^\times)^2\), in agreement with
Ionascu's parametrization of regular tetrahedra in \(\Z^3\)~\cite{Ionascu2009};
related questions for the Platonic solids in \(\Z^3\) were studied by Ionascu
and Markov~\cite{IonascuMarkov2011}.

Several adjacent constructions also fit into the picture. Since multiplying
all coordinates by a common denominator converts a rational realization into an
integral one while multiplying the squared edge length by a rational square,
the rational problem is naturally a square-class refinement of the
corresponding lattice-realizability problem. When \(d=n\), regular
\(d\)-simplices whose vertices are vertices of a cube are equivalent to
Hadamard matrices of order \(d+1\); in particular, the Sylvester construction
gives examples when \(d+1\) is a power of \(2\). This connection appears
already in the classical literature and in later work of Medyanik, Grigoriev,
and Markov~\cite{Schoenberg1937,Pelling1977,Medyanik1973,Grigoriev1980,
Markov2011}. A different but related literature concerns Heronian triangles
and tetrahedra, where one prescribes integral edge lengths together with
rational or integral content and asks for lattice embeddings
\cite{Yiu2001,Lunnon2012,MarshallPerlis2013}. Our approach is also inspired by
the recent work of Bernert and Reinhold~\cite{BernertReinhold2026} on the
analogous problem for hypercubes; see also \cite{BidermanElkies2014}.

The full squared-edge-length classification presented here does not seem to
have appeared previously. In particular, beyond the classical
codimension-\(0\) existence results, the codimension-\(1\) families for
\(d\equiv 0,2\pmod 4\) and the codimension-\(2\) families governed by the local
sets \(\mathcal H_s\) and \(\mathcal U_s\) do not seem to have been isolated
previously.

The next section states the classification precisely. The following section
proves the result via case-by-case analysis of positive definite complements
with prescribed determinant class and local Hasse invariants. We then give
illustrative examples and conclude with structural remarks.

\section{Main Theorem}

Let \(1\le d\le n\). We first fix the notation used in the classification.

A regular \(d\)-simplex with vertices \(v_0,\dots,v_d\in\Q^n\) and squared edge
length \(m\) has Gram matrix
\[
\bigl((v_i-v_0)\cdot(v_j-v_0)\bigr)_{1\le i,j\le d}
=
\frac{m}{2}A_d,
\]
where
\[
A_d=
\begin{pmatrix}
2&1&\cdots&1\\
1&2&\cdots&1\\
\vdots&\vdots&\ddots&\vdots\\
1&1&\cdots&2
\end{pmatrix}.
\]
Thus, writing \(m=2a\), the problem is to decide when the form \(aA_d\)
embeds into the Euclidean form \(n\cdot\langle 1\rangle\).

Write \(s=\sfpart(d+1)\) for the squarefree part of \(d+1\). For a place \(v\)
of \(\Q\), let \((\ ,\ )_v\) denote the Hilbert symbol on \(\Q_v^\times\). For
squarefree \(t\in\Z\), define
\[
N_t^+=\{x^2-ty^2\in\Q_{>0}:x,y\in\Q\}.
\]
For \(t\neq 1\), \(N_t^+\) is the set of positive norms from the quadratic
field \(\Q(\sqrt t)\); for \(t=1\), we have \(N_1^+=\Q_{>0}\). In particular,
\(N_{-1}^+\) is the set of positive rationals that are sums of two rational
squares.

For squarefree \(s>0\), define
\begin{gather*}
\mathcal H_s
=
\{a\in\Q_{>0}:(a,s)_v=1 \text{ for every place \(v\) with \(-s\in(\Q_v^\times)^2\)}\},\\
\mathcal U_s
=
\{a\in\Q_{>0}:(a,-1)_v=1 \text{ for every place \(v\) with \(-as\in(\Q_v^\times)^2\)}\};
\end{gather*}
these conditions depend only on the square class of \(a\). Since \(a>0\) and \(s>0\),
neither \(-s\) nor \(-as\) is a square in \(\R^\times\), so the real place does
not contribute to the defining conditions for \(\mathcal H_s\) and
\(\mathcal U_s\).

For \(X\subseteq\Q_{>0}\), we write \(2X=\{2x:x\in X\}\).

\begin{theorem}\label{thm:main}
Consider \(d\) with \(1\le d\le n\), and put \(s=\sfpart(d+1)\). Then
\(S_{\Q}(d,n)\) depends only on the codimension \(n-d\), the residue class of
\(d\bmod 4\), and \(s\). The only obstructions occur in codimensions \(0\),
\(1\), and \(2\); for \(n-d\ge 3\) we have \(S_{\Q}(d,n)=\Q_{>0}\). More
precisely, we have:

\vspace{6pt}
\begin{center}
\renewcommand{\arraystretch}{1.4}%
\setlength{\tabcolsep}{7pt}%
\begin{tabular}{@{}lcccc@{}}
\toprule
& \(d\equiv 0\pmod 4\)
& \(d\equiv 1\pmod 4\)
& \(d\equiv 2\pmod 4\)
& \(d\equiv 3\pmod 4\) \\
\midrule
\(n-d=0\)
& \makecell[c]{\(\Q_{>0},\ s=1\) \\[0.35ex] \(\varnothing,\ s>1\)}
& \makecell[c]{\(2s(\Q^\times)^2,\ s\in N_{-1}^+\) \\[0.35ex]
                \(\varnothing,\ s\notin N_{-1}^+\)}
& \(\varnothing\)
& \(2s(\Q^\times)^2\) \\[1.1ex]
\(n-d=1\)
& \(2N_s^+\)
& \(2N_{-1}^+\)
& \(2N_{-s}^+\)
& \(\Q_{>0}\) \\
\(n-d=2\)
& \(2\mathcal H_s\)
& \(2\mathcal U_s\)
& \(\Q_{>0}\)
& \(\Q_{>0}\) \\
\(n-d\ge 3\)
& \(\Q_{>0}\)
& \(\Q_{>0}\)
& \(\Q_{>0}\)
& \(\Q_{>0}\). \\
\bottomrule
\end{tabular}
\end{center}
\vspace{4pt}
\end{theorem}

Thus the rational theory stabilizes sharply in codimension \(3\). The first
three rows in Theorem~\ref{thm:main} record the only cases in which the
complement has rank too small to absorb all local invariants automatically.

\begin{remark*}
Let \(S(d,n)\) denote the set of squared edge lengths of regular
\(d\)-simplices with vertices in \(\Z^n\). Then
\[
m\in S_{\Q}(d,n)
\iff
mq^2\in S(d,n)\text{ for some \(q\in\Q^\times\)}.
\]
Equivalently, \(S(d,n)\) and \(S_{\Q}(d,n)\) have the same image in
\(\Q_{>0}/(\Q^\times)^2\). Thus Theorem~\ref{thm:main} determines the image of
\(S(d,n)\) in \(\Q_{>0}/(\Q^\times)^2\); the exact integral classification,
however, requires additional arithmetic within each square class, which we do
not pursue here.
\end{remark*}

\begin{remark*}
The codimension-\(2\) conditions in Theorem~\ref{thm:main} are effectively
finite once \(a\) is fixed. For an odd prime \(p\), the Hilbert symbol of two
\(p\)-adic units is \(1\). Hence \((a,s)_p=1\) whenever \(v_p(a)=0\) and
\(p\nmid s\), and \((a,-1)_p=1\) whenever \(v_p(a)=0\). Moreover, if \(p\mid s\)
and \(v_p(a)=0\), then \(-as\) has odd \(p\)-adic valuation, so it is not a
square in \(\Q_p^\times\). Hence:
\begin{itemize}
\item \(a\in\mathcal H_s\) if and only if \((a,s)_p=1\) for each prime \(p\) in
the finite set
\[
\{2\}\cup\{p:\, v_p(a)\neq 0\}\cup\{p:\, p\mid s\}
\]
for which \(-s\in(\Q_p^\times)^2\);
\item \(a\in\mathcal U_s\) if and only if \((a,-1)_p=1\) for each prime \(p\) in
the finite set
\[
\{2\}\cup\{p:\, v_p(a)\neq 0\}
\]
for which \(-as\in(\Q_p^\times)^2\).
\end{itemize}
The real place never contributes because \(a,s>0\). Thus, although
\(\mathcal H_s\) and \(\mathcal U_s\) are defined over all places, for each
fixed \(a\), membership reduces to finitely many local checks.
\end{remark*}

\section{Proof of Theorem~\ref{thm:main}}\label{sec:proof}

\subsection{Sketch of argument}

Writing \(m=2a\) and \(c=n-d\), the embedding problem is equivalent to finding
a positive definite rational form \(r\) of rank \(c\) such that
\[
aA_d\perp r\simeq n\cdot\langle 1\rangle.
\]
The identity
\[
A_d\perp\langle d+1\rangle\simeq(d+1)\cdot\langle1\rangle
\]
then determines the required determinant class and local Hasse invariants of
\(r\). The proof proceeds by analyzing whether such an \(r\) exists in ranks
\(c\ge 3\), \(c=2\), \(c=1\), and \(c=0\).

\subsection{Preliminaries}

We begin with the standard invariants of rational quadratic forms. For a
diagonal form \(q=\langle a_1,\dots,a_r\rangle\) and a place \(v\) of \(\Q\),
let
\[
\varepsilon_v(q)=\prod_{1\le i<j\le r}(a_i,a_j)_v
\]
denote the local Hasse invariant. For a general form, we compute after
diagonalization; the result is independent of the chosen diagonalization. We
also use the identity
\[
\varepsilon_v(q_1\perp q_2)
=
\varepsilon_v(q_1)\varepsilon_v(q_2)
(\det q_1,\det q_2)_v
\]
for diagonal forms. Quadratic forms over \(\Q\) are classified by their
localizations at all places; equivalently, by dimension, determinant class, the
signature at the real place, and the local Hasse invariants, subject to the
usual product formula; see, e.g., \cite[Ch.~2]{Pfister1995} or
\cite[Ch.~VI]{Lam2005}.

\subsection{Reduction}\label{sec:reduction}

Put \(a=m/2\) and \(c=n-d\).

\begin{proposition}\label{prop:reduction}
We have \(m\in S_{\Q}(d,n)\) if and only if there exists a rational quadratic
form \(r\) of rank \(c\)---positive definite if \(c>0\) and the unique
\(0\)-dimensional form if \(c=0\)---such that
\begin{equation}
aA_d\perp r\simeq n\cdot\langle 1\rangle.
\label{eq:reduction}
\end{equation}
\end{proposition}

\begin{proof}
Assume first that \(m\in S_{\Q}(d,n)\). After translation, we may suppose that
\(v_0=0\). Let \(U=\operatorname{span}_{\Q}(v_1,\dots,v_d)\subseteq \Q^n\). The
Gram matrix of \(v_1,\dots,v_d\) is \(aA_d\). Since \(A_d\) is positive
definite, \(aA_d\) is nonsingular---and hence the vectors \(v_1,\dots,v_d\) are
linearly independent. Hence \(\dim U=d\), so \(\dim U^\perp=n-d=c\). The
orthogonal complement \(U^\perp\) is defined by rational linear equations,
hence has a basis over \(\Q\); because the ambient form
\(n\cdot\langle 1\rangle\) is positive definite, its restriction to \(U^\perp\)
is again positive definite, or the unique \(0\)-dimensional form when \(c=0\).
Therefore
\[
aA_d\perp r\simeq n\cdot\langle 1\rangle
\]
for some such \(r\).

Conversely, assume that such an \(r\) exists, and choose an isometry
\[
aA_d\perp r \xrightarrow{\sim} n\cdot\langle 1\rangle.
\]
Let \(w_1,\dots,w_d\in\Q^n\) be the images of a basis of the \(aA_d\)-summand
with Gram matrix \(aA_d\). Then
\[
|w_i|^2=w_i\cdot w_i=2a=m
\]
for each \(i\), and for \(i\ne j\),
\[
|w_i-w_j|^2
=
w_i\cdot w_i+w_j\cdot w_j-2w_i\cdot w_j
=
2a+2a-2a
=
m.
\]
Thus
\[
0,w_1,\dots,w_d\in\Q^n
\]
form a regular \(d\)-simplex of squared edge length \(m\). Since the Gram
matrix \(aA_d\) is nonsingular, the vectors \(w_1,\dots,w_d\) are linearly
independent, so \(0,w_1,\dots,w_d\) are affinely independent.
\end{proof}

We next record the elementary identity
\begin{equation}\label{eq:Ad-basic}
A_d\perp \langle d+1\rangle \simeq (d+1)\cdot \langle 1\rangle.
\end{equation}
Indeed, in \((d+1)\cdot\langle 1\rangle\), the vectors
\[
u_i=e_i-e_{d+1}\qquad (1\le i\le d)
\]
have Gram matrix \(A_d\), while
\[
w=e_1+\cdots+e_{d+1}
\]
is orthogonal to all \(u_i\) and has norm \(d+1\). Moreover,
\[
e_{d+1}
=
\frac{1}{d+1}\left(w-\sum_{i=1}^d u_i\right),
\qquad
e_i=u_i+e_{d+1}\quad(1\le i\le d),
\]
so \(u_1,\dots,u_d,w\) form a basis of \(\Q^{d+1}\), proving
\eqref{eq:Ad-basic}. Also, \(A_d=I_d+J_d\), where \(J_d\) is the all-\(1\)'s
matrix; hence the eigenvalues of \(A_d\) are \(1\) with multiplicity \(d-1\)
and \(d+1\) with multiplicity \(1\), so
\begin{equation}\label{eq:det-Ad}
\det(A_d)=d+1.
\end{equation}

Write \(d+1=su^2\) with \(u\in\Z_{>0}\). Since
\[
\langle a(d+1)\rangle \simeq \langle as\rangle,
\]
scaling \eqref{eq:Ad-basic} gives
\[
aA_d\perp \langle as\rangle \simeq (d+1)\cdot\langle a\rangle.
\]
Accordingly, the condition \eqref{eq:reduction} in
Proposition~\ref{prop:reduction} is equivalent to the existence of a positive
definite rational form \(r\) of rank \(c\), or the unique \(0\)-dimensional form
if \(c=0\), such that
\begin{equation}\label{eq:cancelled}
(d+1)\cdot\langle a\rangle \perp r
\simeq
(d+c)\cdot\langle 1\rangle \perp \langle as\rangle.
\end{equation}
Indeed, we obtain \eqref{eq:cancelled} by adjoining \(\langle as\rangle\) to
the isometry in Proposition~\ref{prop:reduction}; conversely, canceling
\(\langle as\rangle\) (via Witt cancellation for nondegenerate quadratic forms
over \(\Q\)) recovers the proposition.

Taking determinants in Proposition~\ref{prop:reduction} and using
\eqref{eq:det-Ad}, we obtain
\begin{equation}\label{eq:delta}
\det(r)\equiv a^d s \pmod{(\Q^\times)^2}.
\end{equation}
(Here, we also use the fact that every element of
\(\Q^\times/(\Q^\times)^2\) is its own inverse.)

Assume now that \(c\ge 1\). Comparing Hasse invariants in
\eqref{eq:cancelled} yields
\begin{equation}\label{eq:r-hasse}
\varepsilon_v(r)
=
(a,-1)_v^{d(d+1)/2}(a,s)_v^{d+1}
\end{equation}
for every place \(v\).

Indeed,
\[
\varepsilon_v\bigl((d+1)\cdot\langle a\rangle\bigr)
=
(a,a)_v^{\binom{d+1}{2}}
=
(a,-1)_v^{d(d+1)/2},
\]
while
\[
\varepsilon_v\bigl((d+c)\cdot\langle 1\rangle\perp \langle as\rangle\bigr)
=
1.
\]
Moreover,
\[
(\det((d+1)\cdot\langle a\rangle),\det r)_v
=
(a^{d+1},a^d s)_v
=
(a,a)_v^{d(d+1)}(a,s)_v^{d+1}
=
(a,s)_v^{d+1},
\]
since \((a,a)_v=(a,-1)_v\) and \(d(d+1)\) is even. Formula
\eqref{eq:r-hasse} follows.

Define
\begin{equation}
\eta_v=(a,-1)_v^{d(d+1)/2}(a,s)_v^{d+1}.
\label{eq:eta}
\end{equation}
Then \eqref{eq:r-hasse} is exactly the condition
\(\varepsilon_v(r)=\eta_v\) for all \(v\). The family \(\{\eta_v\}_v\) has
finite support, satisfies \(\eta_\infty=1\), and obeys
\(\prod_v\eta_v=1\) by Hilbert reciprocity.

Conversely, if \(c\ge 1\) and \(r\) is a positive definite rational quadratic
form of rank \(c\) satisfying \eqref{eq:delta} and \eqref{eq:r-hasse}, then the
two sides of \eqref{eq:cancelled} have the same dimension, determinant class,
local Hasse invariants, and positive signature at the real place; hence, they
are isometric over \(\Q\). Thus, for \(c\ge 1\), conditions \eqref{eq:delta}
and \eqref{eq:r-hasse} are both necessary and sufficient.

\subsection{Auxiliary local facts}

If \(r=\langle x,\delta/x\rangle\), then
\begin{equation}\label{eq:binary-hasse}
\varepsilon_v(r)=(x,-\delta)_v.
\end{equation}
Indeed,
\[
1=(x,1)_v=(x,xx^{-1})_v=(x,x)_v(x,x^{-1})_v,
\]
so \((x,x^{-1})_v=(x,x)_v\); hence,
\[
(x,\delta/x)_v
=
(x,\delta)_v(x,x^{-1})_v
=
(x,\delta)_v(x,x)_v
=
(x,\delta)_v(x,-1)_v
=
(x,-\delta)_v.
\]

Next, for squarefree \(t\neq 1\) and \(a\in\Q_{>0}\), the condition
\(a\in N_t^+\) is equivalent to \((a,t)_v=1\) for all places \(v\); this is the
Hasse norm theorem for the quadratic extension \(\Q(\sqrt t)/\Q\), with the
real condition already encoded by the positivity of \(a\). For \(t=1\), we have
\(N_1^+=\Q_{>0}\), since every \(r\in\Q_{>0}\) can be written as
\[
r
=
\left(\frac{r+1}{2}\right)^2
-
\left(\frac{r-1}{2}\right)^2.
\]

We shall also use the prescribed Hilbert-symbol theorem in the following form:
for fixed \(b\in\Q^\times\), a family of signs \(\{\iota_v\}_v\) with finite
support and \(\prod_v\iota_v=1\) is realized as \((x,b)_v=\iota_v\) by some
\(x\in\Q^\times\) if and only if \(\iota_v=1\) at every place where
\(b\in(\Q_v^\times)^2\); see \cite[Ch.~2]{Pfister1995} or
\cite[Ch.~VI]{Lam2005}.

Finally, over a nonarchimedean local field \(\Q_v\), every choice of dimension
\(c\ge 3\), determinant class, and Hasse invariant is realized by some
quadratic form; this follows from the local classification of quadratic spaces
\cite[Ch.~VI]{Lam2005}.

\subsection{A parity lemma}

\begin{lemma}\label{lem:parity}
For \(d\bmod 4\), the parities of the exponents appearing in the Hasse
computations are
\[
\renewcommand{\arraystretch}{1.05}
\setlength{\tabcolsep}{9pt}
\begin{array}{c|cc}
d\bmod 4 & (d+1)\bmod 2 & \frac{d(d+1)}{2}\bmod 2 \\
\hline
0 & 1 & 0 \\
1 & 0 & 1 \\
2 & 1 & 1 \\
3 & 0 & 0.
\end{array}
\]
\end{lemma}

\begin{proof}
The result is immediate upon writing \(d=4k+\ell\) with \(0\le \ell\le 3\) and
reducing the two expressions modulo \(2\).
\end{proof}

\begin{remark*}
For later reference, combining \eqref{eq:delta}, \eqref{eq:r-hasse}, and
Lemma~\ref{lem:parity} yields the following bookkeeping for the complement
\(r\) when \(c\ge 1\): for every place \(v\),
\[
\renewcommand{\arraystretch}{1.06}
\setlength{\tabcolsep}{8pt}
\begin{array}{c|cc}
d\bmod 4 & \det(r)\bmod(\Q^\times)^2 & \varepsilon_v(r) \\
\hline
0 & s  & (a,s)_v \\
1 & as & (a,-1)_v \\
2 & s  & (a,-s)_v \\
3 & as & 1.
\end{array}
\]
Thus the later casework is exactly the problem of realizing these determinant
classes and local Hasse targets by a positive definite form of rank \(c=n-d\).
\end{remark*}

\subsection{Case analysis by codimension}

\subsubsection*{Codimension \(\ge 3\)}

Let \(c\ge 3\), let \(\delta>0\) be a representative of the square class
\(a^d s\), and let \(\eta_v\) be as in \eqref{eq:eta}. The family
\(\{\eta_v\}_v\) has finite support, satisfies \(\eta_\infty=1\), and obeys
\(\prod_v\eta_v=1\).

For each finite place \(v\), the local existence theorem for quadratic spaces
over \(\Q_v\) gives a \(c\)-dimensional nondegenerate form with determinant
class \(\delta(\Q_v^\times)^2\) and Hasse invariant \(\eta_v\). At the real
place, take
\[
r_\infty=(c-1)\cdot\langle 1\rangle\perp \langle \delta\rangle,
\]
which is positive definite, has determinant class \(\delta(\R^\times)^2\), and
has Hasse invariant \(1=\eta_\infty\).

By the Hasse--Minkowski existence and classification theorem, these local
invariants determine a rational quadratic form \(r\) whose localization at each
place has the prescribed dimension, determinant class, and Hasse invariant.
Since \(r_\infty\) is positive definite, the real localization of \(r\) is
positive definite; hence \(r\) is positive definite as a rational form.
Therefore \(r\) satisfies \eqref{eq:delta} and \eqref{eq:r-hasse}, and so
\eqref{eq:cancelled} follows.

\subsubsection*{Codimension \(2\)}

Now let \(c=2\), and let \(\delta>0\) be a representative of the determinant
class \(a^d s\). Every positive definite binary form over \(\Q\) diagonalizes,
so any such form of determinant class \(\delta\) is \(\Q\)-isometric to
\(\langle u,v\rangle\) with \(u,v\in\Q_{>0}\) and
\(uv\in\delta(\Q^\times)^2\). Absorbing the square factor into one coefficient
shows that it is \(\Q\)-isometric to
\[
r=\langle x,\delta/x\rangle
\qquad (x\in\Q_{>0}).
\]
Then \eqref{eq:binary-hasse} and \eqref{eq:r-hasse} become
\begin{equation}\label{eq:c2-master}
(x,-\delta)_v=\eta_v
\qquad\text{for every place \(v\)}.
\end{equation}
By the prescribed Hilbert-symbol theorem, \eqref{eq:c2-master} is solvable for
some \(x\in\Q^\times\) if and only if
\[
\eta_v=1\qquad\text{at every place \(v\) with \(-\delta\in(\Q_v^\times)^2\)}.
\]
Since \(\eta_\infty=1\), the real equation forces \(x>0\). Indeed,
\(\delta>0\), so \(-\delta<0\), and over \(\R\) we have
\((x,-\delta)_\infty=1\) if and only if \(x>0\). Hence
\(r=\langle x,\delta/x\rangle\) is positive definite. Thus the
codimension-\(2\) case reduces to a single local condition.

If \(d\equiv 0\pmod 4\), then \(\delta=s\) and \(\eta_v=(a,s)_v\). Therefore
\eqref{eq:c2-master} is solvable if and only if
\[
(a,s)_v=1
\qquad\text{for every place \(v\) with \(-s\in(\Q_v^\times)^2\)},
\]
that is, if and only if \(a\in\mathcal H_s\).

If \(d\equiv 1\pmod 4\), then \(\delta=as\) and \(\eta_v=(a,-1)_v\). Therefore
\eqref{eq:c2-master} is solvable if and only if
\[
(a,-1)_v=1\qquad\text{for every place \(v\) with \(-as\in(\Q_v^\times)^2\)},
\]
that is, if and only if \(a\in\mathcal U_s\).

If \(d\equiv 2\pmod 4\), then \(\delta=s\) and \(\eta_v=(a,-s)_v\). The
condition is automatic, because \(\eta_v=1\) whenever \(-s\) is a local square;
explicitly, \(x=a\) satisfies \eqref{eq:c2-master}.

If \(d\equiv 3\pmod 4\), then \(\delta=as\) and \(\eta_v=1\). Again the
condition is automatic; explicitly, \(x=1\) satisfies \eqref{eq:c2-master}.

In each case we obtain a positive definite binary form \(r\) with the required
determinant class and local Hasse invariants, so \eqref{eq:cancelled} follows;
this establishes the row \(n-d=2\).

\subsubsection*{Codimension \(1\)}

Let \(c=1\). Choose a positive representative \(\delta\) of the square class
\(a^d s\), and take
\[
r=\langle \delta\rangle.
\]
Since \(\varepsilon_v(r)=1\), equation \eqref{eq:r-hasse} becomes
\[
1=\eta_v
\qquad\text{for every place \(v\)};
\]
by Lemma~\ref{lem:parity}, this is equivalent to
\[
\begin{cases}
(a,s)_v=1& d\equiv 0\pmod 4,\\
(a,-1)_v=1& d\equiv 1\pmod 4,\\
(a,-s)_v=1& d\equiv 2\pmod 4,\\
\text{no condition}& d\equiv 3\pmod 4;
\end{cases}
\]
corresponding to
\[
\begin{cases}
a\in N_s^+& d\equiv 0\pmod 4,\\
a\in N_{-1}^+& d\equiv 1\pmod 4,\\
a\in N_{-s}^+& d\equiv 2\pmod 4,\\
a\in \Q_{>0}& d\equiv 3\pmod 4.
\end{cases}
\]
Since \(r\) has rank \(1\) and is positive definite, the sufficiency statement
proved in Section~\ref{sec:reduction} applies: matching dimension, determinant
class, local Hasse invariants, and positive real signature gives the required
isometry~\eqref{eq:cancelled}; as \(m=2a\), this is exactly the row \(n-d=1\).

\subsubsection*{Codimension \(0\)}

Finally, let \(c=0\). Then \eqref{eq:cancelled} reads as
\[
(d+1)\cdot\langle a\rangle
\simeq
d\cdot\langle 1\rangle\perp \langle as\rangle.
\]
The determinant condition is
\[
a^d s\in (\Q^\times)^2,
\]
and the Hasse condition is
\[
(a,-1)_v^{d(d+1)/2}=1
\]
for every place \(v\). Since both sides are positive definite of rank \(d+1\),
these conditions are also sufficient. Lemma~\ref{lem:parity} yields four cases.

If \(d\equiv 0\pmod 4\), then \(a^d\) is a square, so the determinant condition
forces \(s=1\). The Hasse condition is automatic because \(d(d+1)/2\) is even.
Hence \(S_{\Q}(d,d)=\Q_{>0}\) if \(s=1\), and is empty otherwise.

If \(d\equiv 1\pmod 4\), then the determinant condition is
\(a\in s(\Q^\times)^2\). Writing \(a=sq^2\), the Hasse condition becomes
\((s,-1)_v=1\) for every \(v\), that is, \(s\in N_{-1}^+\). In that case
\(m=2a\in 2s(\Q^\times)^2\).

If \(d\equiv 2\pmod 4\), then \(d\) is even, so the determinant condition would
force \(s=1\). This is impossible because \(d+1\equiv 3\pmod 4\), whereas a
perfect square is congruent only to \(0\) or \(1\pmod 4\).

If \(d\equiv 3\pmod 4\), then the Hasse condition is automatic, and the
determinant condition is \(a\in s(\Q^\times)^2\), equivalently
\(m\in 2s(\Q^\times)^2\).

Thus we have confirmed the row \(n-d=0\), completing the proof of
Theorem~\ref{thm:main}. \qedsymbol

\section{Examples}

\begin{proposition}\label{prop:obstruction}
There is no regular \(4\)-simplex with rational vertices in \(\Q^4\).
\end{proposition}

\begin{proof}
Here \(d=n=4\) and \(s=\sfpart(5)=5\). The codimension-\(0\) row of
Theorem~\ref{thm:main} gives \(S_{\Q}(4,4)=\varnothing\), since \(s>1\);
hence, no such simplex exists.
\end{proof}

\begin{example}\label{ex:positive}
There exists a regular $4$-simplex of squared edge length $10$ in $\Q^5$.
(By Proposition~\ref{prop:obstruction}, no such simplex exists in $\Q^4$.)
\end{example}

\begin{proof}
For \((d,n)=(4,5)\) we have \(s=5\), and Theorem~\ref{thm:main} gives
\(S_{\Q}(4,5)=2N_5^+\). As
\[
5=5^2-5\cdot 2^2,
\]
we have \(5\in N_5^+\), hence \(10\in 2N_5^+\).

An explicit example is given by the row vectors of
\[
W=
\begin{pmatrix}
1&3&0&0&0\\
2&1&1&0&2\\
2&1&-2&0&1\\
\frac54&\frac54&-\frac14&\frac52&\frac34
\end{pmatrix}.
\]
A direct computation yields
\[
WW^{\!T}=5A_4;
\] 
indeed, each row has squared norm \(10\), and each pair of distinct rows has dot
product \(5\).

Thus, if \(v_0=(0,0,0,0,0)\) and \(v_1,\dots,v_4\in\Q^5\) are the rows of
\(W\), then the Gram matrix of \(v_1-v_0,\dots,v_4-v_0\) is \(5A_4\).
Equivalently, all pairwise squared distances among \(v_0,\dots,v_4\) are equal
to \(10\).
\end{proof}

\section{Remarks}

Theorem~\ref{thm:main} depends only on the codimension \(n-d\), the residue
class of \(d\bmod 4\), and the squarefree part \(s=\sfpart(d+1)\). The
parameter \(s\) appears directly in the codimension-\(0\) square classes,
through the norm groups \(N_{\pm s}^+\) in codimension \(1\), and through the
local sets \(\mathcal H_s\) and \(\mathcal U_s\) in codimension \(2\). For each
fixed \(d\), the stable row shows that every positive rational occurs as the
squared edge length of a regular \(d\)-simplex with vertices in
\(\Q^{d+3}\). This is sharper, in the special case of regular simplices,
than Maehara's general bound \cite[Theorem~2]{Maehara1995}; for \(d=2\), the
stronger stabilization at codimension \(2\) is consistent with Beeson's theorem
\cite{Beeson1992}.

Meanwhile, the following proposition explains the parameter dependence more
structurally, along the lines of \cite{BernertReinhold2026}. When
\(\sfpart(d+1)=\sfpart(d'+1)\), the forms \(A_d\) and \(A_{d'}\) differ only by
an orthogonal sum of copies of \(\langle 1\rangle\). If in addition
\(d\equiv d'\pmod 4\), then the corresponding scaled summand is isometric to
the same number of copies of \(\langle 1\rangle\), yielding stabilization of
\(S_{\Q}(d,n)\).

\begin{proposition}\label{prop:stabilization}
Let \(d,d'\ge 1\) satisfy \(\sfpart(d+1)=\sfpart(d'+1)\), and suppose that
\(d'\ge d\). Then
\[
A_{d'}\simeq A_d\perp (d'-d)\cdot\langle 1\rangle
\]
over \(\Q\).

If in addition \(d\equiv d'\pmod 4\), then
\[
S_{\Q}(d,n)=S_{\Q}(d',n+d'-d)
\]
for every \(n\ge d\).
\end{proposition}

\begin{proof}
Write \(d+1=su^2\) and \(d'+1=s(u')^2\), where
\(s=\sfpart(d+1)=\sfpart(d'+1)\). By \eqref{eq:Ad-basic},
\[
A_d\perp \langle s\rangle
\simeq
A_d\perp \langle d+1\rangle
\simeq
(d+1)\cdot\langle 1\rangle,
\]
and similarly
\[
A_{d'}\perp \langle s\rangle
\simeq
(d'+1)\cdot\langle 1\rangle.
\]
Hence,
\[
A_{d'}\perp \langle s\rangle
\simeq
(d'+1)\cdot\langle 1\rangle
\simeq
(d'-d)\cdot\langle 1\rangle \perp (d+1)\cdot\langle 1\rangle
\simeq
A_d\perp (d'-d)\cdot\langle 1\rangle \perp \langle s\rangle.
\]
Witt cancellation then gives
\begin{equation}
A_{d'}\simeq A_d\perp (d'-d)\cdot\langle 1\rangle,
\label{eq:wittAd}
\end{equation}
as claimed.

Now set \(k=d'-d\). If \(d\equiv d'\pmod 4\), then \(k\equiv 0\pmod 4\). For
every \(a\in\Q_{>0}\), we claim that
\begin{equation}\label{eq:k-scaled-euclidean}
k\cdot\langle a\rangle \simeq k\cdot\langle 1\rangle.
\end{equation}
If \(k=0\), then \eqref{eq:k-scaled-euclidean} is immediate. If \(k>0\), then
the two forms have the same rank and are both positive definite over \(\R\).
Their determinant classes agree because \(a^k\in(\Q^\times)^2\). Moreover, for
every place \(v\),
\[
\varepsilon_v(k\cdot\langle a\rangle)
=
(a,a)_v^{\binom{k}{2}}
=
(a,-1)_v^{\binom{k}{2}}.
\]
Since \(k\equiv0\pmod4\), the integer \(\binom{k}{2}=k(k-1)/2\) is even; hence
\[
\varepsilon_v(k\cdot\langle a\rangle)
=
1
=
\varepsilon_v(k\cdot\langle 1\rangle)
\]
for every \(v\). The Hasse--Minkowski classification therefore gives
\eqref{eq:k-scaled-euclidean}.

Suppose now that \(m=2a\in S_{\Q}(d,n)\). By
Proposition~\ref{prop:reduction}, there exists a positive definite rational
form \(r\) of rank \(n-d\) such that
\[
aA_d\perp r\simeq n\cdot\langle 1\rangle.
\]
Then, using \eqref{eq:wittAd} after scaling by \(a\), and then
\eqref{eq:k-scaled-euclidean}, we have
\[
aA_{d'}\perp r
\simeq
aA_d\perp k\cdot\langle a\rangle\perp r
\simeq
aA_d\perp k\cdot\langle 1\rangle\perp r
\simeq
(n+k)\cdot\langle 1\rangle,
\]
so \(m\in S_{\Q}(d',n+k)\).

Conversely, if \(m=2a\in S_{\Q}(d',n+k)\), then for some positive definite
rational form \(r\) of rank \(n-d\) we have
\[
aA_{d'}\perp r\simeq (n+k)\cdot\langle 1\rangle.
\]
Using \eqref{eq:wittAd} after scaling by \(a\), and then
\eqref{eq:k-scaled-euclidean}, we have
\[
aA_{d'}
\simeq
aA_d\perp k\cdot\langle a\rangle
\simeq
aA_d\perp k\cdot\langle 1\rangle.
\]
Thus
\[
aA_d\perp k\cdot\langle 1\rangle\perp r
\simeq
aA_{d'}\perp r
\simeq
(n+k)\cdot\langle 1\rangle.
\]
By Witt cancellation,
\[
aA_d\perp r\simeq n\cdot\langle 1\rangle;
\]
hence, \(m\in S_{\Q}(d,n)\).
\end{proof}

\providecommand{\bysame}{\leavevmode\hbox to3em{\hrulefill}\thinspace}
\providecommand{\MR}{\relax\ifhmode\unskip\space\fi MR }
\providecommand{\MRhref}[2]{%
  \href{http://www.ams.org/mathscinet-getitem?mr=#1}{#2}
}
\providecommand{\href}[2]{#2}

\end{document}